\documentclass[11pt]{article}
\usepackage{enumitem}
\usepackage{natbib}
\usepackage{amsmath,amsthm,amssymb}
\usepackage{graphicx,psfrag,epsf}
\usepackage{epstopdf}
\usepackage{epsfig}
\usepackage{subcaption}
\usepackage{bbm}	
\usepackage{color}
\usepackage{mathtools} 
\usepackage{authblk}
\usepackage{tikz}
\usetikzlibrary{arrows}

\RequirePackage[hyperindex,breaklinks,colorlinks,citecolor=blue,urlcolor=black]{hyperref} 

\setitemize{itemsep=-1mm}
\setenumerate{itemsep=-1mm}

\theoremstyle{plain}
\newtheorem{lemma}{Lemma}
\newtheorem{theorem}{Theorem}

\theoremstyle{definition} 
 
\newtheorem{example}{Example}
\newtheorem{remark}{Remark}

\newcommand{\blind}{0}

\makeatletter
\newcommand{\mylabel}[2]{#2\def\@currentlabel{#2}\label{#1}}
\makeatother

\begin{document}

\def\spacingset#1{\renewcommand{\baselinestretch}%
{#1}\small\normalsize} \spacingset{1}


\if0\blind
{
      \title{\bf
       Posterior contraction in group sparse logit models for categorical responses
  }
\author{Seonghyun Jeong
}
\affil{
Department of Statistics and Data Science,  Department of Applied Statistics\\
 Yonsei University, Seoul, Korea \smallskip \\ \texttt{sjeong@yonsei.ac.kr}}
  \maketitle
} \fi

\if1\blind
{
  \bigskip
  \bigskip
  \bigskip
  \begin{center}
    {\LARGE\bf Title}
\end{center}
  \medskip
} \fi

\begin{abstract}
This paper studies posterior contraction rates in multi-category logit models with priors incorporating group sparse structures. 
We consider a general class of logit models that includes the well-known multinomial logit models as a special case. Group sparsity is useful when predictor variables are naturally clustered and particularly useful for variable selection in the multinomial logit models.
We provide a unified platform for posterior contraction rates of group-sparse logit models that include binary logistic regression under individual sparsity.
No size restriction is directly imposed on the true signal in this study. In addition to establishing the first-ever contraction properties for multi-category logit models under group sparsity, this work also refines recent findings on the Bayesian theory of binary logistic regression.
\end{abstract}


\spacingset{1.0}

\section{Introduction}

The theory of high-dimensional sparse regression has recently received a great deal of attention in the Bayesian community.
Most existing studies on Bayesian sparse regression have examined continuous response variables \citep[e.g.,][]{castillo2015bayesian,martin2017empirical,gao2015general,belitser2017empirical,jeong2020unified}.
However, discrete response variables are also very useful and essential in many areas of application; thus, they deserve far more attention than they have received.
In particular, the theory of Bayesian high-dimensional regression for multi-categorical (nominal) responses has not yet been investigated in the literature.

In this paper, we aim to fill this gap by considering high-dimensional logit models for categorical responses under group sparsity.
For every $i=1,\dots,n$, with the sample size $n$, let the response variable be $Z_i\in\{0,1,\dots,m-1\}$, where $m\ge2$ represents the number of categories.
Let $d$ be the total number of parameters, $X_i\in\mathbb{R}^{(m-1)\times d}$ be a design matrix for the $i$th observation, and $\beta\in\mathbb{R}^{d}$ be a vector of regression coefficients.
We can then write a general logit model for the categorical response $Z_i$ as
\begin{align}
\log \frac{\mathbb{P}(Z_i=\ell)}{\mathbb{P}(Z_i=0)} = X_{i(\ell)}^T\beta, \quad \ell=1,\dots, m-1, \quad i=1,\dots,n,
\label{c3eqn:model}
\end{align}
where $X_{i(\ell)}\in\mathbb R^{d}$ is the $\ell$th row of $X_i$ and $\mathbb{P}$ is the probability operator. The covariate vector $X_{i(\ell)}$ quantifies characteristics of category $\ell$ against the reference category $0$.
It is obvious that the model subsumes logistic regression models for binary response variables. More precisely, model \eqref{c3eqn:model} is reduced to a standard logistic regression model when $m=2$.
Form \eqref{c3eqn:model} is general in the sense that the covariates can vary with $\ell$, but it is often assumed that these covariates are not category-specific. We present the following two examples to elaborate upon this point.

\begin{example}[Variable selection in multinomial logit models]
The right-hand side of \eqref{c3eqn:model} often has the simpler form  $Z_i^T\alpha_\ell$ for some covariates $Z_i\in\mathbb{R}^p$ and parameters $\alpha_\ell\in\mathbb{R}^p$ with $p>0$, in which case the resulting regression model is called a multinomial logit model.
For this model, the covariate $Z_i$ for the $i$th individual is not choice-specific but rather common to all categories, and the likelihoods of the categories are discriminated by the choice-specific parameters $\alpha_\ell$. Common examples for $Z_i$ are intrinsic characteristics of individuals, such as age and gender.
The model can still be put in the general form of \eqref{c3eqn:model} by writing
	$X_i= I_{m-1} \otimes  Z_i^T\in\mathbb{R}^{(m-1)\times p(m-1)} $, $\beta=(\alpha_1^T,\dots,\alpha_{m-1}^T)^T \in\mathbb{R}^{p(m-1)}$, and $d=p(m-1)$, where $\otimes$ denotes the Kronecker product and $I_r$ is the $r\times r$ identity matrix.
	Suppose that we are interested in variable selection for $Z_i$ in the high-dimensional scenarios where sparsity is necessarily incorporated for sensible estimation.
	In this situation, it makes sense for the parameters that are linked to the same covariate to be included or excluded together. This task can be handled by group-level sparsity.
	\label{c3exm:mul}
\end{example}

\begin{example}[Group selection in conditional logit models]
	The general logit model in \eqref{c3eqn:model} is often called a conditional logit model 
	\citep{mcfadden1973conditional}. 
	For this general framework, the covariate $X_{i(\ell)}\in\mathbb R^{d}$ is choice-specific because it calibrates characteristics of category $\ell$ for individual $i$ against the reference category $0$.
	The model is particularly useful in many observational studies and decision sciences where choice-specific data are available. 
	For example, in the analysis of the remarriage and welfare choices of divorced or separated
	women \citep{hoffman1988multinomial}, for the three response categories	(remarriage, remaining single and receiving welfare, remaining single without receiving welfare), the after-tax wage rate and the non-labor income of a woman are different across the categories, meaning that these are choice-specific covariates.
	For the high-dimensional conditional logit models, individual-level sparsity is a natural treatment, but group sparsity may still be of interest, depending on the data and research questions, especially when predictor variables are naturally clustered, as is the case in gene expression data \citep{meier2008group}. 
	\label{c3exm:con}
\end{example}

In view of Example~\ref{c3exm:mul}, group sparse modeling is extremely useful for variable selection in the multinomial logit models. 
However, Example~\ref{c3exm:con} suggests that a specific treatment of the multinomial logit models may not be sufficient and indicates that  considering the general framework itself in \eqref{c3eqn:model} could be highly beneficial.
 We refer the reader to \citet{hoffman1988multinomial} for further discussion on the multinomial and conditional logit models.

We study the posterior contraction rates of model \eqref{c3eqn:model} under group sparsity, possibly with unequal group sizes.  We are primarily interested in the high-dimensional setting for which $p>n$, where $p$ is the number of groups. Clearly, $p\le d$. Note that $p=d$ if sparsity is imposed at the individual level only.
Using a lasso-type penalty, the idea of group sparse estimation was first considered for linear models in \citet{yuan2006model} and extended to logistic regression in \citet{meier2008group}.
A group lasso for multinomial logit models was considered in \citet{vincent2014sparse}.
However, even when taking the frequentist perspective,
theoretical studies on high-dimensional group sparse estimation are mostly directed at linear models 
\citep{nardi2008asymptotic,huang2010benefit,lounici2011oracle},
and few extensions have been attempted; see \citet{blazere2014oracle} for some findings for the generalized linear model setting.
Within the Bayesian framework, the estimation properties for group sparse modeling have only recently been studied, even in the case of linear regression \citep{ning2018bayesian,bai2019spike,gao2015general}.
To the best of our knowledge, the estimation properties for model \eqref{c3eqn:model} with group sparsity have not been examined previously, not even in the frequentist literature.

Although model \eqref{c3eqn:model} has not been scrutinized under group sparsity conditions, some Bayesian works on binary logistic regression, which is subsumed by our setup, do exist. Under the high-dimensional generalized linear model framework, \citet{jiang2007bayesian} established contraction rates relative to the Hellinger metric with sparsity-inducing priors. More recently, \citet{jeong2020posterior} obtained $\ell_q$-type posterior contraction results directly on regression coefficients under relaxed assumptions. \citet{wei2019contraction} examined posterior contraction in logistic regression using continuous shrinkage priors. 
Model selection consistency of high-dimensional logistic regression was considered by \citet{narisetty2019skinny} under individual sparsity and by \citet{lee2020bayesian} under group sparsity, respectively.
All these works, however, require some size restrictions on the true regression coefficients. Such a requirement is often undesirable in high-dimensional scenarios \citep{castillo2015bayesian}. 
To the best of our knowledge, \citet{atchade2017contraction} is the only available Bayesian work that makes no direct restriction on size. He obtained a lasso-type $\ell_2$-contraction rate in high-dimensional logistic regression under certain compatibility conditions. However, we find that his results can be refined under our framework, as will be seen in Section~\ref{c3sec:rates}. As such, this study improves the findings of \citet{atchade2017contraction} and goes beyond it by studying posterior contraction for model \eqref{c3eqn:model} under group sparsity without any direct size restrictions on the coefficients.

The rest of this paper is organized as follows. Section~\ref{c3sec:setup} describes the notation and specifies the prior distribution. Section~\ref{c3sec:rates} provides our main results on the posterior contraction rates of high-dimensional logit models under group sparsity. The technical proofs are provided in Section~\ref{sec:mainproof}. Lastly, Section~\ref{sec:disc} concludes with a discussion. Auxiliary results are presented in Appendix.

\section{Setup and prior specification}
\label{c3sec:setup}

\subsection{Notation}
\label{c3sec:not}
For sequences $a_n$ and $b_n$, 
$a_n\lesssim b_n$ (or $b_n\gtrsim a_n$) means that $a_n\le C b_n$ for some constant $C>0$ independent of $n$, and $a_n\asymp b_n$ means that $a_n\lesssim b_n\lesssim a_n$. 
The entire design matrix is denoted by  $X=(X_1^T,\dots,X_n^T)^T\in\mathbb{R}^{n(m-1)\times d}$.
We assume that the group subsets
$G_1,\dots,G_p$ form a partition of $\{1,\dots,d\}$ in such a manner that $\cup_{j=1}^p G_j =\{1,\dots,d\}$, allowing them to represent which variable is included in which group. We let $g_j$ represent the cardinality of $G_j$, i.e., $g_j=|G_j|$, and write $\overline g = \max_{1\le j\le p} g_j$.
For each $j=1,\dots,p$, let $\beta_j\in\mathbb{R}^{g_j}$ be the subvector of $\beta\in\mathbb{R}^{d}$ whose elements are chosen by $G_j$. Similarly, we define $X_{\cdot j}\in\mathbb{R}^{n(m-1)\times g_j}$, $j=1,\dots,p$, to be submatrices of $X\in\mathbb{R}^{n(m-1)\times d}$, where the columns of $X_{\cdot j}$ are chosen by $G_j$. Let $\beta_0$ denote the true value of $\beta$, from which the observations are generated.

For a vector $\beta\in\mathbb R^d$ and a set $S\subset\{1,\dots,p\}$ of group indices, we write $\beta_S =\{\beta_j , j \in S \}$ and $\beta_{S^c} =\{\beta_j , j \notin S \}$ to separate $\beta$ into zero and nonzero coefficients using $S$.
We also denote by $S_\beta=\{j:\beta_j\ne 0_{g_j}\} \subset \{1,\dots ,p \}$ the effective group index determined by $\beta$, where $0_{g_j}$ is the $g_j$-dimensional zero vector. The cardinalities of $S$ and $S_\beta$ are denoted by $s = |S|$ and $s_\beta = |S_\beta|$, respectively.
In particular, the group index of the true parameter $\beta_0$ and its cardinality are written as $S_0$ and $s_0$, respectively. 
We let $d_S = \sum_{j\in S} g_j$ denote the dimension of $\beta_S$, and write $d_0 = d_{S_0}$ for the true dimension. 

Let $\lVert\cdot\rVert_2$ denote the $\ell_2$-norm of a vector. For a $d$-dimensional vector $\beta$, we write $\lVert \beta\rVert_{2,1}=\sum_{j=1}^p\lVert \beta_j\rVert_2$ to denote the $\ell_{2,1}$-norm that is typically used in the context of group sparsity.
Although not specified, one can easily see that $\lVert\cdot\rVert_{2,1}$ depends on the group subsets $G_1,\dots,G_p$.
 Slightly abusing notation, we also write $\lVert \beta_S\rVert_{2,1}=\sum_{j\in S}\lVert \beta_j\rVert_2$, which depends only on $G_j$, $j\in S$.
For a matrix $X$ with $d$ columns, we define the matrix norm: 
\begin{align*}
\lVert X\rVert_\ast=\max_{1\le j\le p} \lVert X_{\cdot j}\rVert_{\rm sp},
\end{align*}
where $\lVert \cdot \rVert_{\rm sp}$ is the spectral norm of the matrix.
This expression is a natural generalization of the norm, which is the square root of the maximum diagonal entry of $X^T X$, widely used for individual sparse inference in the literature
\citep[e.g.,][]{castillo2015bayesian,belitser2017empirical}.
Note that our definition of $\lVert X\rVert_\ast$ is reduced to that norm if $\overline g=1$.
For a vector or matrix, we denote by $\lVert \cdot \rVert_\infty$ the max-norm, the maximum element of an object in absolute value.

We define the multinomial response variable $Y_{i\ell} = \mathbbm{1}(Z_i=\ell)$, $i=1,\dots,n$, $\ell=1,\dots,m-1$, such that  for any $i$, $\sum_{\ell=1}^{m-1} Y_{i\ell}=1$ if $Z_i>0$ and $\sum_{\ell=1}^{m-1} Y_{i\ell}=0$ otherwise. In what follows, we work with the response vector $Y=(Y_1^T,\dots,Y_n^T)^T\in\mathbb R^{n(m-1)}$, where $Y_i=(Y_{i1},\dots, Y_{i,m-1})^T\in\mathbb R^{m-1}$, $i=1,\dots,n$.
We write the density of $Y$ with respect to a dominating counting measure as $f_\beta^n$ for an arbitrary parameter $\beta$ and as $f_0^n$ for the true parameter $\beta_0$, respectively.
The notations $\mathbb P_0$ and  $\mathbb E_0$ denote the probability and expectation operators under the true model with $\beta_0$, respectively.
We also let $\mu=\mathbb E_0 Y$ and $W=\mathbb E_0 \{(Y-\mu)(Y-\mu)^T\}$ be the expected value and the covariance matrix, respectively, of $Y$ under the true model.

Some conditions on the design matrix $X$ are required for estimation of the high-dimensional regression coefficients.
We first define the following compatibility number:
\begin{align*}
\phi(S)=\inf\!\Bigg\{\frac{\lVert W^{1/2} X\beta \rVert_2 \sqrt{s} }{\lVert X\rVert_\ast \lVert\beta\rVert_{2,1}} : \lVert \beta_{S^c} \rVert_{2,1}\le 7\lVert \beta_S  \rVert_{2,1}, \beta_S\ne 0\Bigg \}.
\end{align*}
The constant 7 is of no particular interest and can be replaced with modifications of the constants appearing in our main results.
To recover the $\ell_{2,1}$- and $\ell_2$-contraction rates, we also define the ($W$-adjusted) uniform compatibility number and the smallest scaled singular value, respectively, as
\begin{align*}
\psi_1(s)&=\inf\!\Bigg\{\frac{\lVert W^{1/2} X\beta \rVert_2 \sqrt{s_\beta} }{\lVert X\rVert_\ast \lVert\beta\rVert_{2,1}}:1\le s_\beta\le s\Bigg\},\quad\psi_2(s)=\inf\!\Bigg\{\frac{\lVert W^{1/2}X\beta \rVert_2 }{\lVert X\rVert_\ast \lVert\beta\rVert_2}:1\le s_\beta\le s\Bigg\}.
\end{align*}
The definitions of $\phi$, $\psi_1$, and $\psi_2$ are modified from the compatibility conditions in \citet{castillo2015bayesian} in such a manner that they are suited for our logit models under group sparsity. 
More precisely, our $\phi$ and $\psi_1$ are defined with the $\ell_{2,1}$-norm for group sparse inference, whereas those in \citet{castillo2015bayesian} are defined with the $\ell_1$-norm.
The covariance matrix $W$ is also inserted to account for the non-quadratic likelihood ratio.
If we plug in the identity matrix for $W$ while imposing individual sparsity, then our definitions correspond to the compatibility conditions given in \citet{castillo2015bayesian} up to constants (see Remark~\ref{rmk:compat} below). By the Cauchy-Schwarz inequality, it follows that $\psi_2(s)\le \psi_1(s)$ for every $s\ge 1$. It is also easy to see that all our compatibility constants above  are bounded. This can easily be verified by evaluating them with a unit vector and the maximal eigenvalue of $W$ (see the proof of Lemma~\ref{lmm:lrbound} in Section~\ref{sec:mainproof}).

\begin{remark}[Alternative to $\phi$]
Our definition of the compatibility number $\phi$ is not directly reduced to that of \citet{castillo2015bayesian} even when individual sparsity is imposed and $W$ is the identity matrix, due to the fact that the sparse vector $\beta_S$ is used instead in the denominator of the ratio in \citet{castillo2015bayesian}. Along the same lines, our $\phi$ can be accordingly modified as
\begin{align*}
\phi_{\rm mod}(S)=\inf\!\Bigg\{\frac{\lVert W^{1/2} X\beta \rVert_2 \sqrt{s} }{\lVert X\rVert_\ast \lVert\beta_S\rVert_{2,1}} : \lVert \beta_{S^c} \rVert_{2,1}\le 7\lVert \beta_S  \rVert_{2,1}, \beta_S\ne 0\Bigg \}.
\end{align*}
It is trivial that $\phi_{\rm mod}(S)/8\le \phi(S)\le \phi_{\rm mod}(S)$ for every $S$, meaning that the two coefficients are essentially identical up to constants. It is not difficult to see that all our results established in this paper can be rendered with $\phi_{\rm mod}$ by modifying the appearing constants accordingly.
  \citet{atchade2017contraction} also defined his compatibility number in a manner similar to ours. We use $\phi$ rather than $\phi_{\rm mod}$ to compare our main results with those in that work fairly.
	\label{rmk:compat}
\end{remark}

\begin{remark}[Asymptotic order of $\lVert X \rVert_\ast$]
	The asymptotic behavior of $\lVert X \rVert_\ast$ is important for understanding how our compatibility conditions perform. Understanding this behavior is also essential, as $\lVert X \rVert_\ast$ appears in the main results on posterior contraction; see Theorem~\ref{c3thm:contraction} below. If $\overline g$ is bounded, then $\lVert X \rVert_\ast$ is typically of order $\sqrt{n}$ in usual regression settings, which can be easily verified by the inequality $\lVert A \rVert_{\rm F}/\sqrt{r}\le \lVert A \rVert_{\rm sp}\le \lVert A \rVert_{\rm F}$ for a matrix $A$ of rank $r$, where $\lVert \cdot \rVert_{\rm F}$ denotes the Frobenius norm.
	Although not as clearly as in the case of bounded $\overline g$, this asymptotic behavior may still hold even when $\overline g$ tends to infinity. For example, if each row of $X$ is independently drawn from a sub-Gaussian distribution, we still have $\lVert X\rVert_\ast\asymp\sqrt{n}$ with high probability; see Lemma~\ref{lmm:ratex} in Appendix. Thus, collinearity among the covariates may not affect the order of $\lVert X\rVert_\ast$ unless it approaches the perfect linearity as $n\rightarrow \infty$.
	\label{rmk:xorder}
\end{remark}

\subsection{Prior specification}
\label{c3sec:prior}
A prior distribution should be carefully designed to obtain the desired posterior contraction rate.
As is customary in individual sparse regression, we first select a group dimension $s$ from a prior distribution $\pi_p(s)$, and then randomly choose a group index set $S\subset \{1,\dots,p\}$ for given $s$. The nonzero part $\beta_S$ of the coefficients is then selected from a continuous prior density $h_S$ on $\mathbb{R}^{d_S}$ while $\beta_{S^c}$ is set to zero. The resulting prior distribution for $(S,\beta)$ is summarized as
\begin{align*}
(S,\beta)\mapsto \pi_p(s)\binom{p}{s}^{-1} h_S(\beta_S) \delta_0(\beta_{S^c}),
\end{align*}
where $\delta_0$ is the Dirac measure at zero on $\mathbb{R}^{d-d_S}$.

It remains to specify $\pi_p$ and $h_S$. For the prior $\pi_p$ on the group size, we consider a prior distribution such that for some constants $A_1, A_2, A_3, A_4>0$,
\begin{align}
\frac{A_1}{\max\{p, n^{\overline g}\}^{A_3}}  \le\frac{\pi_p(s)}{\pi_p(s-1)}\le \frac{A_2}{\max\{p, n^{\overline g}\}^{A_4}},\quad s=1,\dots,p.
\label{c3eqn:sprior}
\end{align}
This prior distribution is also modified from the one given in \citet{castillo2015bayesian} to suit our group sparse modeling. The term $\max\{p, n^{\overline g}\}$ holds the key to the adaptation to unknown group sparsity.
If $\overline g=1$, i.e., sparsity is imposed only at the individual level, then the prior in \eqref{c3eqn:sprior} is reduced to the one widely used in the high-dimensional setups \citep[e.g.,][]{castillo2015bayesian,martin2017empirical,belitser2017empirical,jeong2020posterior,jeong2020unified}.

For the prior density $h_S$ on the nonzero coefficients, we consider 
\begin{align}
h_S(\beta_S)=\left(\frac{\lambda}{\sqrt{\pi}}\right)^{d_S}  \frac{\prod_{j\in S} \Gamma(g_j/2)}{2^s\prod_{j\in S} \Gamma(g_j)} e^{-\lambda\lVert\beta_S\rVert_{2,1} },\quad \lambda=8\lVert X\rVert_\ast\sqrt{\max\{\log p,\overline g \log n\}}.
\label{c3eqn:nonzeroprior}
\end{align}
The density in \eqref{c3eqn:nonzeroprior} is a product of $s$-fold symmetric Kotz-type distributions \citep{fang1990symmetric}. 
It is easy to see that this density is reduced to a standard Laplace density if $\overline g=1$.
As in the Laplace prior for individual sparse regression in \citet{castillo2015bayesian}, the term $e^{ -\lambda\lVert\beta_S\rVert_{2,1} }$ and the scale parameter $\lambda$ hold the key to obtaining our target rate in group sparse estimation. The constant $8$ in $\lambda$ has no particular meaning and can be replaced with appropriate modifications. For linear regression, note that \citet{castillo2015bayesian} used a wider range of $\lambda$ to allow for decreasing sequences. In our setup with the logit model, however, it is unclear as to whether the  $\lambda$ in \eqref{c3eqn:nonzeroprior} can be weakened.

\section{Posterior contraction rates}
\label{c3sec:rates}
\subsection{Main results}

With the prior distribution $\Pi$ specified in Section~\ref{c3sec:prior},
the posterior distribution $\Pi(\,\cdot\,|Y)$ of $\beta$ is defined by Bayes' rule.
In this section, we study contraction properties of the posterior distribution under suitable assumptions on the design matrix $X$.

We first establish a bound for the effective group dimension, i.e., the number of groups with nonzero coefficients. 
The bound allows us to restrict our attention to models of relatively small size. 
The following theorem shows that the posterior distribution is concentrated on much smaller group dimensions than the full size $p$. 

\begin{theorem}[Effective group size]
	For the logit model in \eqref{c3eqn:model} and the prior specified in Section~\ref{c3sec:prior},
	there exists a constant $M_1>0$ such that for any $M_2>3$,
	\begin{align*}
	\sup_{\beta_0\in\mathcal B_1(M_1)}
	\mathbb{E}_0\Pi\left\{\beta:s_\beta > s_0 + \frac{M_2}{A_4}\left(1+\frac{33}{\phi^2(S_0)}\right)s_0\,\bigg|\,Y \right\}\rightarrow 0,
	\end{align*}
	where $\mathcal B_1(M)=\{\beta_0: s_0 \sqrt{\max\{\log p,\overline g \log n\}} \max_i\lVert X_i\rVert_\ast \le M \phi^2(S_0 ) \lVert X\rVert_\ast \}$.
	\label{c3thm:dimen}
\end{theorem}

As in \citet{castillo2015bayesian} and \citet{atchade2017contraction}, the constants in our threshold are not optimized and hence have no particular meaning. A close examination of the proof reveals that the constants can be substantially improved if the response variable is binary, but we present the results with universal constants for categorical responses for simplicity.
The constants are nevertheless unimportant, as $A_4$ can be chosen to be as large as desired.

We are now ready to examine posterior contraction rates for the regression coefficients.
We first define $\xi_0=s_0+ A_4^{-1}\{4+{100}/{\phi^2(S_0)}\}s_0$ such that most of the posterior mass is concentrated on $s_\beta<\xi_0$ by Theorem~\ref{c3thm:dimen}.
The next theorem shows that the posterior distribution of $\beta$ contracts to $\beta_0$ at the desired rate with respect to the $\ell_{2,1}$- and $\ell_2$-metrics.
While \citet{atchade2017contraction} adopted the general posterior contraction theory with the entropy/testing approach \citep{ghosal2000convergence,ghosal2007convergence}, we deal directly with the expression for the posterior distribution of our logit model, making our proof much simpler while giving rise to faster rates. Still, as in \citet{atchade2017contraction}, our approach to the proof is based on bounds of the likelihood ratio derived from the self-concordant property \citep{bach2010self}; see Section~\ref{sec:pre} for more details.

\begin{theorem}[Posterior contraction]
	For the logit model in \eqref{c3eqn:model} and the prior specified in Section~\ref{c3sec:prior},
	there exist constants $M_3>0$ and $M_4>0$ such that
	\begin{align*}
	\sup_{\beta_0\in \mathcal B_2(M_3) }	
	{\mathbb E}_0\Pi\left\{\beta:\lVert W^{1/2} X(\beta-\beta_0)\rVert_2> \frac{M_4\sqrt{s_0\max\{\log p,\overline g \log n\} } }{\psi_1(\xi_0+s_0)\phi(S_0)}\,\bigg |\,Y \right\}&\rightarrow 0,\\
	\sup_{\beta_0\in \mathcal B_2(M_3) }	
	{\mathbb E}_0\Pi\left\{\beta:\lVert \beta-\beta_0\rVert_2> \frac{M_4\sqrt{s_0\max\{\log p,\overline g \log n\}}}{\psi_1(\xi_0+s_0)\psi_2(\xi_0+s_0)\phi(S_0)\lVert X\rVert_\ast} \,\bigg|\,Y \right\}&\rightarrow 0,\\
	\sup_{\beta_0\in \mathcal B_2(M_3) }	
	{\mathbb E}_0\Pi\left\{\beta:\lVert \beta-\beta_0\rVert_{2,1}>\frac{M_4s_0\sqrt{\max\{\log p,\overline g \log n\}}}{\psi_1^2(\xi_0+s_0)\phi^2(S_0)\lVert X\rVert_\ast} \,\bigg|\,Y \right\}&\rightarrow 0,
	\end{align*}
	where $\mathcal B_2(M)=\{\beta_0: s_0 \sqrt{\max\{\log p,\overline g \log n\}} \max_i\lVert X_i\rVert_\ast \le M \psi_1^2( \xi_0 +s_0)\phi^2(S_0) \lVert X\rVert_\ast \}$.
	\label{c3thm:contraction}
\end{theorem}

We comment on the obtained rates. To our knowledge, the minimax risk bounds for our setup have not been discovered previously in the literature, but the bounds for related settings are still useful to surmise the optimality of our results. 
Assume that $\lVert X\rVert_\ast\asymp \sqrt{n}$ as in Remark~\ref{rmk:xorder}.
With the exception of the compatibility conditions, our $\ell_2$-rate matches the minimax rate of group sparse linear regression with equal group sizes up to logarithmic factors \citep{lounici2011oracle}. 
Our rates also substantially refine the estimation rates established by \citet{blazere2014oracle} for generalized linear models with a group lasso.
Under the Bayesian framework, \citet{gao2015general} recently obtained the minimax posterior contraction in group sparse linear regression using elliptical priors. The Gram matrix $X^T X$ is incorporated into their prior to cancel some terms out nicely, but it is unclear as to whether the same approach can be used for our logit model, as the likelihood ratio is not quadratic. On the other hand, \citet{ning2018bayesian} obtained contraction rates comparable to ours for group sparse linear regression with unknown variance.

It is worth noting that our results are greatly simplified when $\overline g$ is bounded.
One particularly interesting example is the variable selection problem for the multinomial logit models in Example~\ref{c3exm:mul}, where $\overline g=m-1$.
This situation also includes the case where sparsity is imposed at the individual level only, i.e., $\overline g=1$.
 In the case of bounded $\overline g$, the term $\overline g \log n$ is removed from the results since $n< p$. Moreover, due to the relation $\lVert X\rVert_\infty\le \max_i\lVert X_i\rVert_\ast\le \sqrt{(m-1) \overline g}\lVert X\rVert_\infty$, the term $\max_i\lVert X_i\rVert_\ast$ can be replaced by $\lVert X\rVert_\infty$ in such a scenario.

\begin{remark}[Bounded compatibility constants]
Since the compatibility constants in our rates obscure the interpretation of the obtained rates, it may be of interest to establish the conditions under which they can be removed. In the linear regression setups, it is known that the compatibility constants can be bounded away from zero under mild conditions \citep{van2014higher}.
Due to the additional matrix $W$ appearing in the definitions, however, this is not the case for our logit setup.
Nonetheless, this is still possible under stronger conditions.
For example, if the true linear predictor $\lVert X\beta_0\rVert_\infty$ is known to be bounded such that the smallest eigenvalue of $W$ is bounded away from zero, our compatibility constants can be bounded away from zero under the same conditions as the linear model setups, as, in this case, it follows that $\lVert W^{1/2} X\beta\rVert_2 \gtrsim \lVert  X\beta\rVert_2 $. A stochastic bound on $\lVert X_i\beta_0\rVert_\infty$ would be sufficient; see, for example, Lemma A.4 of \citet{narisetty2019skinny}.
\end{remark}

\begin{remark}[Indirect size restriction on $\beta_0$]
As frequently mentioned above, as in \citet{atchade2017contraction}, our main results do not require direct size restrictions on $\beta_0$ through, say, $\lVert\beta_0\rVert_\infty$ or $\lVert\beta_0\rVert_2$.
This point is one of the main advantages our theory has over other results that hold only on some norm-bounded subsets \citep[e.g.,][]{wei2019contraction,narisetty2019skinny,lee2020bayesian}. Nonetheless, we should point out that some restrictions are indirectly rendered through the sets $\mathcal B_1$ and $\mathcal B_2$ in our main results. More specifically, both of these sets depend on the true group size $s_0$. Although the uniformity is restricted in such a manner, doing so is still allowable since  high-dimensional regression coefficients are often assumed to be sparse enough for sensible estimation. Indeed, our condition on $\mathcal B_2$ is very similar to holding the $\ell_{2,1}$-consistency.
  It is also trivial that our conditions are related to the true linear predictor $X\beta_0$ in a indirect manner, as our compatibility constants involve the matrix $W$ in their definitions. Notwithstanding these underlying limitations, our conditions are weaker than those of \citet{atchade2017contraction} (see Section~\ref{sec:comparison} below), not to mention other works relying on stronger norm-bounded subsets.
\end{remark}

\subsection{Comparison to \citet{atchade2017contraction} when $m=2$ and $\overline g=1$}
\label{sec:comparison}
Our modeling framework is reduced to binary logistic regression under individual sparsity when $m=2$ and $\overline g=1$. 
\citet{atchade2017contraction} used the same prior as ours in studying the contraction rates of high-dimensional logistic regression, so it is naturally of interest to compare our results with those established there. In fact, our Theorem~\ref{c3thm:dimen} and Theorem~\ref{c3thm:contraction} refine the results of Theorem~4 and Remark~5 in \citet{atchade2017contraction} under relaxed conditions. We now elucidate this point.

We reproduce the results of Theorem 4 in \citet{atchade2017contraction} using our notation.
Similar to our Theorem~\ref{c3thm:dimen}, \citet{atchade2017contraction} obtained a result for effective dimension
with the threshold $\tilde\xi_0=s_0+c_0\{1+{n\lVert X\rVert_\infty^2}/{(\lVert X\rVert_\ast^2 \phi^2(S_0))}+c_n\}s_0$ for some constant $c_0>0$ and possibly increasing sequence $c_n>0$.  Since $\lVert X\rVert_\ast\le \sqrt{\overline g n (m-1)}\lVert X \rVert_\infty=\sqrt{n }\lVert X \rVert_\infty$ for $m=2$ and $\overline g=1$, this threshold is clearly larger than the one we give in Theorem~\ref{c3thm:dimen}. 
In particular, our threshold is free of $c_n$, coming from the additional compatibility condition used in \citet{atchade2017contraction}, which can possibly cause a deterioration in the rate. 
Moreover, the $\ell_2$-contraction rate established by \citet{atchade2017contraction} is given by
\begin{align}
\frac{\sqrt{n}\lVert X\rVert_\infty \sqrt{\tilde\xi_0 \log p}}{ \psi_2^2(s_0+\tilde\xi_0) \lVert X\rVert_\ast^2} \asymp \max\left\{ \frac{\sqrt{n}\lVert X\rVert_\infty}{\lVert X\rVert_\ast \phi(S_0)}, \sqrt{c_n} \right\}\frac{\sqrt{n}\lVert X\rVert_\infty \sqrt{s_0 \log p}}{ \psi_2^2(s_0+\tilde\xi_0) \lVert X\rVert_\ast^2}.
\label{eqn:l2l2}
\end{align}
One can easily see that this rate is worse than our $\ell_2$-rate given in Theorem~\ref{c3thm:contraction}, due to the inequalities $\lVert X\rVert_\ast\le \sqrt{n }\lVert X \rVert_\infty$, $\xi_0\lesssim \tilde\xi_0$, and $\psi_2\le \psi_1$. The $\ell_1$-rate given in Remark~5 of \citet{atchade2017contraction} can also be compared to ours in a similar manner.

In addition, our boundedness conditions are weaker than those used in \citet{atchade2017contraction}.
To see this point, observe that the condition for his effective dimension translates into $\sqrt{n}\lVert X\rVert_\infty^2 s_0\sqrt{\log p} \lesssim \phi^2(S_0)\lVert X\rVert_\ast^2 $ (page 2 of the supplement of \citet{atchade2017contraction}). Clearly, this bound is stronger than ours on $\mathcal B_1$ since $\max_i\lVert X_i\rVert_\ast=\lVert X \rVert_\infty$ and $\lVert X\rVert_\ast\le \sqrt{n}\lVert X \rVert_\infty$  if $m=2$ and $\overline g=1$. Similarly, the $\ell_2$-rate condition, which translates into $\sqrt{n}\lVert X\rVert_\infty^2 \tilde \xi_0\sqrt{\log p} \lesssim\psi_2^2(s_0+\tilde\xi_0) \lVert X\rVert_\ast^2$ (page 3 of the supplement of \citet{atchade2017contraction}), is also stronger than our $\mathcal B_2$. This can be easily seen by expanding $\tilde \xi_0$ as in \eqref{eqn:l2l2}.

\section{Proofs of the main results}
\label{sec:mainproof}

\subsection{Preliminaries}
\label{sec:pre}

Here, we first provide intermediate results that are used to prove our main results. The proofs of Theorem~\ref{c3thm:dimen} and Theorem~\ref{c3thm:contraction} are deferred to Section~\ref{sec:proof}.

\subsubsection{Bounds of the likelihood ratio}

As in \citet{atchade2017contraction}, the self-concordant property \citep{bach2010self} holds the key to our approach to the proof. Self-concordant functions have the property that their third derivatives  are controlled by their second derivatives.
As a results, lower and upper Taylor expansions of such functions can be obtained \citep{bach2010self}.  In Lemma~\ref{c3lmm:selfcon} in Appendix, we show that the multi-category logit models in \eqref{c3eqn:model} hold the self-concordant property, thus allowing the construction of the upper and lower bounds for the likelihood ratio given below.

\begin{lemma} 
	 The logit model in \eqref{c3eqn:model} satisfies	
	\begin{align*}
	\frac{(\beta-\beta_0)^T X^T W X (\beta-\beta_0)  }{2+4\max_i\lVert X_i\rVert_\ast\lVert\beta-\beta_0\rVert_{2,1}} &\le (Y-\mu)^T X(\beta-\beta_0)-\log \frac{f_\beta^n}{f_0^n}(Y) 
	\le  \frac{1}{2}(\beta-\beta_0)^T X^T X (\beta-\beta_0).
	\end{align*}
	\label{lmm:lrbound}
\end{lemma}

\begin{proof}	 
	For any $x=(x_1,\dots,x_{m-1})^T\in\mathbb{R}^{m-1}$, define the function ${\sf exp}:\mathbb{R}^{m-1}\mapsto(0,\infty)^{m-1}$ such that $ {\sf exp}(x)=(e^{x_1},\dots,e^{x_{m-1}})^T$. We also write $1_{m-1}$ for the $(m-1)$-dimensional one-vector.
	Now, let $b:\mathbb{R}^{m-1}\mapsto(0,\infty)$ such that $b(\cdot) = \log(1+{\sf exp}(\cdot)^T 1_{m-1})$ and write its gradient vector and Hessian matrix as $\nabla b$ and  $\nabla^2 b$, respectively.
	We let $\theta_i=X_i\beta$ with an arbitrary $\beta$ and $\theta_{0i}=X_i\beta_0$ with the true $\beta_0$.
	We also define the expected value $\mu_i=(1+{\sf exp}(\theta_{0i})^T 1_{m-1})^{-1}{{\sf exp}(\theta_{0i})}$ and the covariance matrix $W_i={\rm diag}(\mu_i)-\mu_i\mu_i^T$ of $Y_i$ under the true model such that $\mu=(\mu_1^T,\dots,\mu_n^T)^T$ and $W$ is  the block-diagonal matrix formed by stacking $W_i$, $i=1,\dots,n$.
	Observe that $\nabla b(\theta_{0i})=\mu_i$ and $\nabla^2 b(\theta_{0i})=W_i$.
	Thus, one can easily check that
	\begin{align}	
	\begin{split}
	(Y-\mu)^T X(\beta-\beta_0)-\log \frac{f_\beta^n}{f_0^n}(Y) &= \sum_{i=1}^n\left\{ \log\frac{1+{\sf exp}(\theta_i)^T 1_{m-1}}{1+{\sf exp}(\theta_{0i})^T 1_{m-1}} -\frac{{\sf exp}(\theta_{0i})^T(\theta_i-\theta_{0i})}{1+{\sf exp}(\theta_{0i})^T 1_{m-1}}\right\}\\
	&=\sum_{i=1}^n\left\{ b(\theta_i)-b(\theta_{0i}) - \nabla b(\theta_{0i})^T (\theta_i-\theta_{0i}) \right\}.
	\end{split}
	\label{eqn:lrbound}
	\end{align}	
	Using Proposition~1 of \citet{bach2010self} and Lemma~\ref{c3lmm:selfcon} in Appendix,
	the display is bounded below by
	\begin{align*}
	&\sum_{i=1}^n \frac{(\theta_i-\theta_{0i})^T [\nabla^2 b(\theta_{0i})](\theta_i-\theta_{0i}) }{16\lVert\theta_i-\theta_{0i}\rVert_2^2} (e^{-4\lVert\theta_i-\theta_{0i}\rVert_2}+4\lVert\theta_i-\theta_{0i}\rVert_2-1)\\
	&\quad \ge \sum_{i=1}^n \frac{(\theta_i-\theta_{0i})^T [\nabla^2 b(\theta_{0i})](\theta_i-\theta_{0i}) }{2+4\lVert\theta_i-\theta_{0i}\rVert_2},
	\end{align*}
	where the inequality holds since $e^{-x}+x-1 \ge x^2/(2+x)$ for every $x\ge0$, verifying the first inequality of the assertion.
	
	By the Taylor expansion, \eqref{eqn:lrbound} is bounded by $(1/2)\sum_{i=1}^n (\theta_i-\theta_{0i})^T [\nabla^2 b(\tilde\theta_i)](\theta_i-\theta_{0i})$ 
	for some $\tilde\theta_i$ that lies between $\theta_i$ and $\theta_{0i}$.
	Observe that $\nabla^2 b(\tilde\theta_i)$ is still a covariance matrix of a multinomial random variable with some parameters. By \citet{watson1996spectral}, one can easily see that $\max_i\lVert\nabla^2 b(\tilde\theta_i)\rVert_{\rm sp}\le 1$, which verifies the second inequality of the assertion. (Since \citet{watson1996spectral} deals with extended multinomial variables for which the sum of the probability vector is $1$, we use the fact that the largest eigenvalue of a principal submatrix is not larger than that of the original matrix.) 
\end{proof}

\subsubsection{Tail probability of $\max_{1\le j\le p}\lVert X_{\cdot j}^T (Y-\mu)\rVert_2$}

Our proof requires a tail probability of $\max_{1\le j\le p}\lVert X_{\cdot j}^T (Y-\mu)\rVert_2$. This is similar in spirit to \citet{castillo2015bayesian} and \citet{atchade2017contraction} being based on such bounds for a scalar version of $X_{\cdot j}^T (Y-\mu)$ under individual sparsity. While in those papers the bounds are trivially obtained by the tail inequality for normal distributions or Hoeffding's inequality, our situation is more complicated as $X_{\cdot j}^T (Y-\mu)$ is a $g_j$-dimensional vector due to the group sparse modeling. Here we formally derive the required tail inequality. Our bound is derived by the tail property of quadratic forms of bounded random vectors, provided in Lemma~\ref{c3lmm:quad} in Appendix.  Similar bounds are also obtainable in other studies on sub-Gaussian vectors \citep[e.g.,][]{hsu2012tail,zajkowski2020bounds,jin2019short}, but we aim here to obtain a bound with a specific constant.

\begin{lemma} For the logit model in \eqref{c3eqn:model} with any $\beta_0\in\mathbb{R}^{d}$,
	\begin{align*}
	{\mathbb P}_0\left\{ \max_{1\le j\le p}\lVert X_{\cdot j}^T (Y-\mu)\rVert_2 > 4 \lVert X \rVert_\ast\sqrt{\max\{\log p,\overline g \}}  \right\} \le  \max\{p, n^{\overline g}\}^{-3/4}.
	\end{align*}
	\label{c3lmm:maxprob}
\end{lemma}

\begin{proof}
	Note that $Y-\mu$ has a bounded support.
	By the Markov inequality followed by Lemma~\ref{c3lmm:quad}, we have that for every $t>0$ and $u< 1/( 4\lVert X_{\cdot j}\rVert_{\rm sp}^2)$,
	\begin{align*}
	{\mathbb P}_0\left\{\lVert X_{\cdot j}^T (Y-\mu) \rVert_2> t\right\}&\le e^{-ut^2} \mathbb{E}_0\exp\left\{ u\lVert  X_{\cdot j}^T (Y-\mu) \rVert_2^2 \right\}\le e^{-ut^2}\exp\left\{ \frac{u \cdot  {\rm tr} (X_{\cdot j} X_{\cdot j}^T) }{1-4 u \lVert X_{\cdot j} \rVert_{\rm sp}^2}  \right\},
	\end{align*}
	for $k=1,\dots,p$.
	Note that ${\rm tr} (X_{\cdot j} X_{\cdot j}^T)\le g_j \lVert  X_{\cdot j} \rVert_{\rm sp}^2$ since the rank of $X_{\cdot j}$ is at most $g_j$. Hence, by choosing $u= 1/ (8\lVert X_{\cdot j}\rVert_{\rm sp}^2)$, the rightmost side of the last display is further bounded by
	\begin{align*}
	\exp\left(-\frac{t^2}{ 8 \lVert X_{\cdot j} \rVert_{\rm sp}^2} +  \frac{g_j}{4} \right)\le \exp\left(-\frac{t^2}{ 8 \lVert X \rVert_\ast^2} +  \frac{\overline g}{4} \right).
	\end{align*}
	Choosing $t=4 \lVert X \rVert_\ast\sqrt{\max\{\log p,\overline g \} }$, we obtain
	\begin{align*}
	{\mathbb P}_0\left\{ \max_{1\le j\le p}\lVert X_{\cdot j}^T (Y-\mu)\rVert_2 > 4 \lVert X \rVert_\ast\sqrt{\max\{\log p,\overline g \} }  \right\}&\le p \exp\left\{-2 (\max\{\log p,\overline g \}) + \overline g/4 \right\}\\
	&\le \max\{p, n^{\overline g}\}^{-3/4} .
	\end{align*}
	This leads to the desired assertion.
\end{proof}

\subsubsection{Lower bound of the denominator of the posterior}
A lower bound for the denominator of the posterior is essential in establishing the posterior contraction rates \citep{ghosal2000convergence,ghosal2007convergence}. Below, we derive a lower bound that gives rise to our target rate.
\begin{lemma} 
	For the model in \eqref{c3eqn:model} and the prior specified in Section~\ref{c3sec:prior}, if $d_0\le n$,
	\begin{align*}
	\int_{\mathbb R^{d}} \frac{f_\beta^n}{f_0^n}(Y) d\Pi(\beta)\ge e^{-1/128} e^{-\lambda\lVert \beta_0 \rVert_{2,1}}\frac{ \pi_p(s_0)}{\max\{p, n^{\overline g}\}^{3s_0}}.
	\end{align*}
	\label{lmm:lbo}
\end{lemma}

\begin{proof}
	Restricting the set to $S=S_0$, note first that
	\begin{align}
	&\int_{\mathbb R^{d}} \frac{f_\beta^n}{f_0^n}(Y) d\Pi(\beta)\ge \frac{\pi_p(s_0)}{\binom{p}{s_0}} \int_{\mathbb R^{d}} \frac{f_\beta^n}{f_0^n}(Y)h_{S_0}(\beta_{S_0})d\beta_{S_0}\otimes\delta(\beta_{S_0^c}).
	\label{eqn:qwbo}
	\end{align}
	Let $X_{S_0}\in\mathbb R^{n(m-1)\times d_0}$ be the submatrix of $X$ with columns chosen by $S_0$. 
	By Lemma~\ref{lmm:lrbound}, the integral term of the preceding display is  bounded below by
	\begin{align*}
	&\int_{\mathbb R^{d_0}} \exp\left\{ (Y-\mu)^T X_{S_0}(\beta_{S_0}-\beta_{0,S_0})-\frac{1}{2}\lVert X_{S_0} (\beta_{S_0}-\beta_{0,S_0})\rVert_2^2\right\} h_{S_0}(\beta_{S_0})d\beta_{S_0}\\
	&\quad\ge e^{-\lambda\lVert \beta_0 \rVert_{2,1}} \int_{\mathbb R^{d_0}} \exp\left\{ (Y-\mu)^T X_{S_0}\beta_{S_0}-\frac{1}{2}\lVert X_{S_0} \beta_{S_0}\rVert_2^2\right\} h_{S_0}(\beta_{S_0})d\beta_{S_0},
	\end{align*}
	where the inequality $h_{S_0}(\beta_{S_0})\ge e^{-\lambda\lVert \beta_{0,S_0} \rVert_{2,1}}h_{S_0}(\beta_{S_0}-\beta_{0,S_0})$ is employed. Following \citet{castillo2015bayesian}, using Jensen's inequality, the integral term in the last display is bounded below by
	\begin{align}
	\int_{\mathbb R^{d_0}} \exp\left\{-\frac{1}{2}\lVert X_{S_0} \beta_{S_0}\rVert_2^2\right\} h_{S_0}(\beta_{S_0})d\beta_{S_0}\ge e^{-1/128} \int_{\lVert X\rVert_\ast \lVert\beta_{S_0}\rVert_{2,1}\le 1/8} h_{S_0}(\beta_{S_0})d\beta_{S_0},
	\label{eqn:lbb1}
	\end{align}
	since $\lVert X_{S_0} \beta_{S_0}\rVert_2^2\le \lVert X_{S_0} \rVert_\ast \lVert\beta_{S_0}\rVert_{2,1}\le \lVert X \rVert_\ast \lVert\beta_{S_0}\rVert_{2,1}$.
	
	Based on our prior for $\beta$, it is not hard to see that $\lVert\beta_j\rVert_2$ has a gamma distribution with rate parameter $g_j$ and scale parameter $\lambda$.
	Since it follows that $\lVert\beta_S\rVert_{2,1}$ has a gamma distribution with rate parameter $d_S$ and scale parameter $\lambda$, using  the Poisson-gamma relationship,
	\begin{align*}
	\int_{\lVert\beta\rVert_{2,1}\le a} h_S(\beta_S)d\beta_S = \sum_{k=d_S}^\infty\frac{(a\lambda)^k e^{-a\lambda}}{k!} \ge  \frac{(a\lambda)^{d_S}e^{-a\lambda}}{d_S!}.
	\end{align*}
	Therefore, \eqref{eqn:lbb1} is bounded below by
	\begin{align*}
	\frac{(\lambda/(8\lVert X\rVert_\ast))^{d_0}e^{-\lambda/(8\lVert X\rVert_\ast)}}{d_0!} \ge \frac{e^{-\sqrt{\max\{\log p,\overline g \log n\}}}}{n^{\overline g s_0}}\ge \frac{1}{\max\{p, n^{\overline g}\} n^{\overline g s_0}},
	\end{align*}
	where the inequality $d_0!\le d_0^{d_0}\le n^{\overline g s_0}$ is utilized. Since $s_0\ge 1$ and $\binom{p}{s_0}\le p^{s_0}$, putting everything together, \eqref{eqn:qwbo} is bounded below by $e^{-1/128} e^{-\lambda\lVert \beta_0 \rVert_{2,1}} \pi_p(s_0)\max\{p, n^{\overline g}\}^{-3s_0}$, which leads to the desired assertion.
\end{proof}

\subsection{Proof of Theorem~\ref{c3thm:dimen} and Theorem~\ref{c3thm:contraction}}
\label{sec:proof}
We are now ready to prove the main results.

\begin{proof}[Proof of Theorem~\ref{c3thm:dimen}]
	Let $\mathcal T_n$ be the event in Lemma~\ref{c3lmm:maxprob}. Define $\mathcal B=\{ \beta: s_\beta> R \}$ for some $R\ge s_0$ to be specified later.
	By Lemma~\ref{c3lmm:maxprob}, we only need to show that $\mathbb E_0\Pi(\mathcal B|Y) \mathbbm 1_{\mathcal T_n}$ tends to zero uniformly over the set given in the theorem, for some appropriately chosen $R$. It is not difficult to see that $\lVert X\rVert_\ast\le \sqrt{n}\max_i\lVert X_i\rVert_\ast$; hence, the set $\mathcal B_1$ is stronger than the condition $d_0\le n$.
	Thus, by Lemma~\ref{lmm:lbo} and Fubini's theorem, it is easy to see that
	\begin{align}
	\mathbb E_0\Pi(\mathcal B|Y) \mathbbm 1_{\mathcal T_n}=\mathbb E_0\frac{\int_{\mathcal B} (f_\beta^n/f_0^n)(Y)\Pi(\beta)}{\int_{\mathbb R^d} (f_\beta^n/f_0^n)(Y)\Pi(\beta)}\mathbbm 1_{\mathcal T_n} \lesssim \frac{\max\{p, n^{\overline g}\}^{3s_0}}{\pi_p(s_0)}
	\int_{\mathcal B} e^{\lambda\lVert \beta_0 \rVert_{2,1}} \mathbb E_0 \frac{f_\beta^n}{f_0^n} \mathbbm 1_{\mathcal T_n}\Pi(\beta).
	\label{eqn:entbo}
	\end{align}
	Note  that the integral term on the right-most side is equal to
	\begin{align}
	\sum_{S:s> R} \frac{\pi_p(s)}{\binom{p}{s}} \left(\frac{\lambda}{\sqrt{\pi}}\right)^{d_S} \frac{\prod_{j\in S} \Gamma(g_j/2)}{2^s\prod_{j\in S} \Gamma(g_j)}\int_{\mathbb R^{d}} \frac{e^{-\lambda\lVert\beta\rVert_{2,1}}}{e^{-\lambda\lVert\beta_0\rVert_{2,1}}}\mathbb E_0  \frac{f_\beta^n}{f_0^n}(Y) \mathbbm 1_{\mathcal T_n} d\beta_S\otimes\delta(\beta_{S^c}) ,
	\label{eqn:intbo}
	\end{align}
	and by Lemma~\ref{lmm:lrbound} and Lemma~\ref{c3lmm:maxprob},
	\begin{align}	
	\log \frac{f_\beta^n}{f_0^n}(Y)\mathbbm 1_{\mathcal T_n}\le \frac{\lambda}{2} \lVert \beta-\beta_0 \rVert_{2,1} - \frac{(\beta-\beta_0)^T X^T W X (\beta-\beta_0) }{2+4\max_i\lVert X_i\rVert_\ast\lVert\beta-\beta_0\rVert_{2,1}}.
	\label{eqn:loglbo}
	\end{align}
	One can easily verify that
	\begin{align}	
	\begin{split}
	\lVert\beta_0\rVert_{2,1}-\lVert\beta\rVert_{2,1}+\frac{1}{2}\lVert\beta-\beta_0\rVert_{2,1} & = 
	\lVert\beta_0\rVert_{2,1}-\lVert\beta\rVert_{2,1}+\frac{1}{2}\lVert\beta_{S_0^c}\rVert_{2,1}+\frac{1}{2}\lVert\beta_{S_0}-\beta_{0,S_0}\rVert_{2,1}\\
	&\le -\frac{1}{2}\lVert\beta_{S_0^c}\rVert_{2,1}+\frac{3}{2}\lVert\beta_{S_0}-\beta_{0,S_0}\rVert_{2,1}.
	\end{split}
	\label{eqn:cal1}
	\end{align}
	If $7\lVert\beta_{S_0}-\beta_{0,S_0}\rVert_{2,1}\le \lVert \beta_{S_0^c} \rVert_{2,1}$, the rightmost side of \eqref{eqn:cal1} is equal to
	$-(1/2)\lVert\beta_{S_0^c}\rVert_{2,1}+(7/4)\lVert\beta_{S_0}-\beta_{0,S_0}\rVert_{2,1}-(1/4)\lVert\beta_{S_0}-\beta_{0,S_0}\rVert_{2,1}\le -(1/4)\lVert\beta-\beta_0\rVert_{2,1}$,
	allowing us to obtain from \eqref{eqn:loglbo} that
	\begin{align*}	
	\frac{e^{-\lambda\lVert\beta\rVert_{2,1}}}{e^{-\lambda\lVert\beta_0\rVert_{2,1}}}\mathbb E_0  \frac{f_\beta^n}{f_0^n}(Y) \mathbbm 1_{\mathcal T_n} 
	&\le\exp\left\{ -\frac{\lambda}{4}\lVert \beta-\beta_0 \rVert_{2,1} \right\}.
	\end{align*}
	If $7\lVert\beta_{S_0}-\beta_{0,S_0}\rVert_{2,1}> \lVert \beta_{S_0^c} \rVert_{2,1} $, since the leftmost side of \eqref{eqn:cal1} is bounded by $(3/2) \lVert\beta-\beta_0\rVert_{2,1}$, we obtain that
	\begin{align*}	
	&\frac{e^{-\lambda\lVert\beta\rVert_{2,1}}}{e^{-\lambda\lVert\beta_0\rVert_{2,1}}}\mathbb E_0  \frac{f_\beta^n}{f_0^n}(Y) \mathbbm 1_{\mathcal T_n}\le \exp\left\{\left(-\frac{\lambda}{4}+\frac{7\lambda}{4}\right)\lVert\beta-\beta_0\rVert_{2,1}- \frac{s_0^{-1}\lVert X\rVert_\ast^2 \lVert\beta-\beta_0\rVert_{2,1}^2 \phi^2(S_0)}{2+4\max_i\lVert X_i\rVert_\ast\lVert\beta-\beta_0\rVert_{2,1}}\right\}.
	\end{align*}
	We now make use of the following fact: for any $x>0,A>0,B>0,C>0$ such that $AC\le (1-\delta) B$ with $\delta\in(0,1)$, 
	\begin{align*}	
	Ax-\frac{Bx^2}{2+Cx}\le Ax-\frac{ABx^2}{2A+(1-\delta) Bx}\le \frac{2A^2x}{2A+(1-\delta) Bx}\le \frac{2A^2}{(1-\delta) B}.
	\end{align*}
	We therefore obtain that on $\mathcal B_1(M_1)$ for some $M_1>0$,
	\begin{align}	
	\frac{7\lambda}{4}\lVert\beta-\beta_0\rVert_{2,1}- \frac{s_0^{-1}\lVert X\rVert_\ast^2 \lVert\beta-\beta_0\rVert_{2,1}^2 \phi^2(S_0)}{2+4\max_i\lVert X_i\rVert_\ast\lVert\beta-\beta_0\rVert_{2,1}}\le \frac{99 s_0\max\{\log p,\overline g \log n\}}{\phi^2(S_0)}.
	\label{eqb:qqza}
	\end{align}
	Hence, for both cases ($7\lVert\beta_{S_0}-\beta_{0,S_0}\rVert_{2,1}\le \lVert \beta_{S_0^c} \rVert_{2,1}$ and $7\lVert\beta_{S_0}-\beta_{0,S_0}\rVert_{2,1} > \lVert \beta_{S_0^c} \rVert_{2,1}$), 
	\begin{align*}	
	\frac{e^{-\lambda\lVert\beta\rVert_{2,1}}}{e^{-\lambda\lVert\beta_0\rVert_{2,1}}}\mathbb E_0  \frac{f_\beta^n}{f_0^n}(Y) \mathbbm 1_{\mathcal T_n}
	\le \exp\left\{-\frac{\lambda}{4}\lVert\beta-\beta_0\rVert_{2,1}+\frac{99 s_0\max\{\log p,\overline g \log n\}}{\phi^2(S_0)}\right\}.
	\end{align*}
	Therefore, \eqref{eqn:intbo} is bounded by
	\begin{align*}
	&\exp\!\left\{\frac{99 s_0\max\{\log p,\overline g \log n\}}{\phi^2(S_0)} \right\}\!\sum_{S:s> R} \!\frac{\pi_p(s)}{\binom{p}{s}} \!\left(\frac{\lambda}{\sqrt{\pi}}\right)^{d_S} \!\frac{\prod_{j\in S} \Gamma(g_j/2)}{2^s\prod_{j\in S} \Gamma(g_j)}\int_{\mathbb R^{d_S}} \!e^{-(\lambda/4)\lVert\beta_S-\beta_{0,S}\rVert_{2,1} } d\beta_S .
	\end{align*}
	Directly evaluating the integral, the summation term becomes
	\begin{align*}
	\sum_{S:s> R} \frac{\pi_p(s)}{\binom{p}{s}} 4^{d_S}&\le \sum_{s=R+1}^p \pi_p(s)4^{s\overline g}\\
	&\le \pi_p(s_0) 4^{s_0\overline g} \left\{\frac{4^{\overline g}A_2}{\max\{p, n^{\overline g}\}^{A_4}}\right\}^{R+1-s_0} \sum_{j=0}^\infty \left\{\frac{4^{\overline g} A_2}{\max\{p, n^{\overline g}\}^{A_4}}\right\}^j.
	\end{align*}
	The series term is bounded for sufficiently large $n$. Hence, we see from \eqref{eqn:entbo} that $	\mathbb E_0\Pi(\mathcal B|Y) \mathbbm 1_{\mathcal T_n}$ is bounded by a constant multiple of
	\begin{align*}
	\exp\bigg\{&\left(3+\frac{99}{\phi^2(S_0)}\right) s_0\max\{\log p,\overline g \log n\} \\
	&+ (R+1-s_0)(\overline g\log 4+\log A_2-A_4\max\{\log p,\overline g \log n\})\bigg\} .
	\end{align*}
	Choosing $R=s_0+M_2A_4^{-1}\{1+33/\phi^2(S_0)\}s_0$ for any $M_2>3$ allows the assertion to be verified.
\end{proof}

\begin{proof}[Proof of Theorem~\ref{c3thm:contraction}]
	Let $\mathcal T_n$ be the event in Lemma~\ref{c3lmm:maxprob} and define $\mathcal B=\{ \beta: s_\beta> \xi_0 , \lVert W^{1/2} X(\beta-\beta_0) \rVert_2 > R \}$ for some $R\ge 0$ to be specified later.
	The boundedness condition on $\mathcal B_2(M)$ is stronger than that in Theorem~\ref{c3thm:dimen}.
	Hence, by Theorem~\ref{c3thm:dimen} and Lemma~\ref{c3lmm:maxprob}, it suffices to show that $\mathbb E_0\Pi(\mathcal B|Y) \mathbbm 1_{\mathcal T_n}$ tends to zero uniformly over the set given in the theorem for some appropriately chosen $R$.
	Observe that as in the proof of Theorem~\ref{c3thm:dimen}, the condition $d_0\le n$ is satisfied on $\mathcal B_2$, meaning that we can apply Lemma~\ref{lmm:lbo}.	
	Using the calculations in \eqref{eqn:entbo} and \eqref{eqn:intbo}, it is easy to see that $\mathbb E_0\Pi(\mathcal B|Y) \mathbbm 1_{\mathcal T_n}$ is bounded by a constant multiple of
	\begin{align*}
	\frac{\max\{p, n^{\overline g}\}^{3s_0}}{\pi_p(s_0)}\sum_{S:s> \xi_0 } \frac{\pi_p(s)}{\binom{p}{s}} \left(\frac{\lambda}{\sqrt{\pi}}\right)^{d_S} \frac{\prod_{j\in S} \Gamma(g_j/2)}{2^s\prod_{j\in S} \Gamma(g_j)}\int_{\mathcal B} \frac{e^{-\lambda\lVert\beta\rVert_{2,1}}}{e^{-\lambda\lVert\beta_0\rVert_{2,1}}}\mathbb E_0  \frac{f_\beta^n}{f_0^n}(Y) \mathbbm 1_{\mathcal T_n} d\beta_S\otimes\delta(\beta_{S^c}) .
	\end{align*}
	Using \eqref{eqn:loglbo}, we obtain that
	\begin{align*}	
	&\frac{e^{-\lambda\lVert\beta\rVert_{2,1}}}{e^{-\lambda\lVert\beta_0\rVert_{2,1}}}\mathbb E_0  \frac{f_\beta^n}{f_0^n}(Y) \mathbbm 1_{\mathcal T_n}\le \exp\Bigg\{\left(-\lambda+\frac{5\lambda}{2}\right)\lVert\beta-\beta_0\rVert_{2,1}- \frac{(\beta-\beta_0)^T X^T W X (\beta-\beta_0)}{2+4\max_i\lVert X_i\rVert_\ast\lVert\beta-\beta_0\rVert_{2,1}}\Bigg\},
	\end{align*}
	since the leftmost side of \eqref{eqn:cal1} is bounded by $(3/2) \lVert\beta-\beta_0\rVert_{2,1}$. Observe that by the definition of $\psi_1$, the exponent in the last expression is bounded by
	\begin{align*}	
	&-\lambda\lVert\beta-\beta_0\rVert_{2,1}+\left(-\lambda+\frac{7\lambda}{2}\right) \frac{ \sqrt{\xi_0+s_0} \lVert W^{1/2} X (\beta-\beta_0)\rVert_2}{\lVert X \rVert_\ast\psi_1(\xi_0+s_0)} \\
	&\quad- \lVert W^{1/2} X (\beta-\beta_0)\rVert_2^2 \bigg/ \left\{2+\frac{4\sqrt{\xi_0+s_0}\max_i\lVert X_i\rVert_\ast\lVert W^{1/2} X (\beta-\beta_0)\rVert_2 }{\lVert X \rVert_\ast\psi_1(\xi_0+s_0)}\right\}.
	\end{align*}
	As in \eqref{eqb:qqza}, there exists a constant $C>0$ such that on $\mathcal B_2(M_3)$ for some $M_3>0$, the last expression is bounded by
	\begin{align*}	
	-\lambda\lVert\beta-\beta_0\rVert_{2,1}-\frac{ \lambda \sqrt{\xi_0+s_0} \lVert W^{1/2} X (\beta-\beta_0)\rVert_2}{\lVert X \rVert_\ast\psi_1(\xi_0+s_0)}+ \frac{C (\xi_0+s_0)\max\{\log p,\overline g \log n\}}{\psi_1^2(\xi_0+s_0)}.
	\end{align*}
	Making use of this bound, we see that
	\begin{align*}	
	\mathbb E_0\Pi(\mathcal B|Y) \mathbbm 1_{\mathcal T_n}&\lesssim \frac{\max\{p, n^{\overline g}\}^{3s_0}}{\pi_p(s_0)}\exp\left\{-\frac{ \lambda \sqrt{\xi_0+s_0} R }{\lVert X \rVert_\ast\psi_1(\xi_0+s_0)}+ \frac{C (\xi_0+s_0)\max\{\log p,\overline g \log n\}}{\psi_1^2(\xi_0+s_0)}\right\} \\
	&\quad\times\sum_{S:s> \xi_0} \frac{\pi_p(s)}{\binom{p}{s}}\left(\frac{\lambda}{\sqrt{\pi}}\right)^{d_S} \frac{\prod_{j\in S} \Gamma(g_j/2)}{2^s\prod_{j\in S} \Gamma(g_j)} \int_{\mathbb R^{d_S}} e^{-\lambda\lVert\beta_S-\beta_{0,S}\rVert_{2,1} } d\beta_S.
	\end{align*}
	The summation term is bounded by 1. 
 It can be shown that $\psi_1(\xi_0+s_0)\le 1$ by plugging in the unit vector and noting that $\lVert W\rVert_{\rm sp}\le 1$  (see the proof of Lemma~\ref{lmm:lrbound}).
	Choose $R=M_4 \sqrt{(\xi_0+s_0)\max\{\log p,\overline g \log n\}}/ \psi_1(\xi_0+s_0)$ for a large enough $M_4>0$.
	Since $\pi(s_0)\ge A_1^{s_0} \max\{p, n^{\overline g}\}^{-A_3s_0}\pi(0)\gtrsim A_1^{s_0} \max\{p, n^{\overline g}\}^{-A_3s_0}$, the first assertion holds if $M_4$ is suitably large. The second and third assertions hold directly by the definitions of $\psi_1$ and $\psi_2$.
\end{proof}

\section{Discussion}
\label{sec:disc}

This paper studies the posterior contraction rates of high-dimensional logit models under group sparsity. Whereas many existing studies on nonlinear models impose some size restrictions on the true regression coefficients, we do not impose such restrictions since they are particularly undesirable in high-dimensional scenarios. Other Bayesian asymptotic properties, such as the Bernstein-von Mises theorem and selection consistency, are also of interest in the high-dimensional regression setups. Unlike for the linear regression models in \citet{castillo2015bayesian}, establishing those properties with our prior is not straightforward due to the restricted range of $\lambda$. \citet{narisetty2019skinny} and \citet{lee2020bayesian} studied selection consistency in Bayesian high-dimensional logit models, though these studies were restricted to binary logistic regression under individual sparsity. Their approaches, however, require direct size restrictions on the regression coefficients. Characterizing additional Bayesian asymptotic properties without such restrictions is an interesting topic for future research.

There is at least one limitation regarding our results, namely, the obtained rates can be deemed suboptimal in the worst-case scenario. Consider a situation in which one group $j_\ast\in\{1,\dots,p\}$, whose corresponding coefficients are zero, grows much faster than other groups such that $\overline g=g_{j_\ast}$ and $g_j=o(g_{j_\ast})$, $j\ne j_\ast$.  Because this group is assumed to be inactive, i.e., $j_\ast\notin S_0$, it is preferable that this group does not change our rates. However, the rates blow up since they are dependent on $\overline g$. Generally, we are more interested in the well-balanced case in which all $g_i$ behave similarly.

\section*{Acknowledgment}
This research was supported by the Yonsei University Research Fund of 2021-22-0032.

\appendix
\section{Appendix: Auxiliary results}

\subsection{Asymptotic behavior of $\lVert X\rVert_\ast$}

\begin{lemma}
	Suppose that each row of $X\in\mathbb R^{n(m-1)\times p}$ is an independent sub-Gaussian vector.
	If $\log p=o(n)$ and $\overline g=o(n)$, then
	$\lVert X\rVert_\ast \asymp \sqrt{n}$ with probability tending to one.
	\label{lmm:ratex}
\end{lemma}	
\begin{proof}
	Observe that by Theorem 5.39 of \citet{vershynin2012introduction}, there exist constants $C_1>0$ and $C_2>0$ such that for any $t>0$, 
	\begin{align*}
	\mathbb P\left\{ \sigma_{\min}(X_{\cdot j})\le \sqrt{n(m-1)}-C_1\sqrt{g_i}-t \right\}&\le e^{-C_2t^2/2},\\
	\mathbb P\left\{\sigma_{\max}(X_{\cdot j})\ge \sqrt{n(m-1)}+C_1\sqrt{g_i}+t\right\}&\le e^{-C_2t^2/2},
	\end{align*} 
	where $\sigma_{\min}(X_{\cdot j})$ and $\sigma_{\max}(X_{\cdot j})$ are the smallest and largest singular values of $X_{\cdot j}$, respectively. Choosing $t=\sqrt{n}/2$, the first line of the display verifies $\lVert X \rVert_\ast\gtrsim \sqrt{n}$ with high probability since $\overline g=o(n)$. Now, observe that
	\begin{align*}
	\mathbb P\left\{\lVert X \rVert_\ast\ge\sqrt{n(m-1)}+C_1\sqrt{\overline g}+t\right\}&\le \sum_{j=1}^p\mathbb  P\left\{\sigma_{\max}(X_{\cdot j})\ge\sqrt{n(m-1)}+C_1\sqrt{g_i}+t\right\}\\
	&\le pe^{-C_2 t^2/2}.
	\end{align*} 
	Choose $t=2\sqrt{(\log p)/C_2}$. Since $\log p=o(n)$ and $\overline g=o(n)$, we have that $\lVert X \rVert_\ast\lesssim \sqrt{n}$ with high probability.
\end{proof}

\subsection{Self-concordant property of multi-category logit models}

\begin{lemma}
	For any $v=(v_1,\dots,v_{m-1})^T\in\mathbb{R}^{m-1}$ and $w=(w_1,\dots,\allowbreak w_{m-1})^T\in\mathbb{R}^{m-1}$, the function $\eta:\mathbb R\mapsto \mathbb R$ defined by $\eta(t)=\log (1+{\sf exp}(w+tv)^T 1_{m-1})$ satisfies $|\eta'''(t)|\le 4\lVert v\rVert_2 \eta''(t)$ for every $t\in\mathbb R$.
	\label{c3lmm:selfcon}
\end{lemma}

\begin{proof}
	By direct calculations, one obtains that
	\begin{align*}
	e^{\eta(t)}\eta'(t) &= \sum_{j=1}^{m-1} v_j e^{w_j+tv_j },\\
	e^{2\eta(t)}\eta''(t) &=\sum_{j=1}^{m-1} v_j^2 e^{w_j+tv_j }+\sum_{j < k} e^{w_j+w_k+t(v_j+v_k)} (v_j-v_k)^2.
	\end{align*} 
	Since $	e^{2\eta(t)}\eta''(t)\ge0$, differentiating both sides of the second line,	
	\begin{align*}
	e^{2\eta(t)}|\eta'''(t)| 
	&\le 2|\eta'(t)|e^{2\eta(t)}\eta''(t)  +\sum_{j=1}^{m-1} |v_j|^3 e^{w_j+tv_j } +\sum_{j < k} e^{w_j+w_k+t(v_j+v_k)} |v_j+v_k|(v_j-v_k)^2\\
	&\le 2|\eta'(t)| e^{2\eta(t)}\eta''(t)  +  	2\lVert v \rVert_2 e^{2\eta(t)}\eta''(t).
	\end{align*} 
	The assertion follows from the display by plugging in the bound
	\begin{align*}
	|\eta'(t)| = \frac{ |\sum_{j=1}^{m-1} v_j e^{w_j+tv_j }|}{ 1+ \sum_{j=1}^{m-1} e^{w_j+tv_j }}\le \lVert v \rVert_2.
	\end{align*}
\end{proof}

\subsection{On quadratic forms of bounded random vectors}


\begin{lemma}
	Let $(Z_j\in \mathbb{R}^{r_j})_{j=1}^n$ be a sequence of independent random vectors such that for every $j\le n$, $\mathbb{E} Z_j = 0$ and $\mathbb{P}\{Z_j\in{\rm supp}(Z_j)\}=1$ for a bounded support ${\rm supp}(Z_j)$ of $Z_j$ (note that for every $j\le n$, the entries in $Z_j$ need not be independent). Let $Z=(Z_1^T,\dots, Z_n^T)^T$. Then for any real positive semidefinite matrix $Q$, we have
	\begin{align*}
	\mathbb{E}\exp\left\{ t Z^T Q Z \right\}\le \exp\left\{ \frac{t \max_{j}\bar b_j^2 {\rm tr}(Q)  }{1-2t \max_j \tilde b_j^2 \lVert Q\rVert_{\rm sp}}  \right\},\quad 0<t< \frac{1}{2 \max_j \tilde b_j^2 \lVert Q\rVert_{\rm sp} },
	\end{align*}
	where for every $j\le n$,
	\begin{align*}
	\bar b_j=\max_{\xi_j\in{\rm supp}(Z_j)} \lVert \xi_j\rVert_2,\quad \tilde b_j=\max_{\xi_j,\xi_j'\in{\rm supp}(Z_j)} \lVert \xi_j-\xi_j'\rVert_2.
	\end{align*}
	\label{c3lmm:quad}
\end{lemma}

\begin{proof}
	We first write $Z^T Q Z = \sum_{1\le j,k\le n} Z_j^T Q_{jk} Z_k$ using the submatrices $Q_{jk}\in \mathbb{R}^{r_j\times r_k}$, $j,k\in\{1,\dots, n\}$, such that 
	\begin{align*} Q=\begin{pmatrix} 
	Q_{11} & \dots & Q_{1n} \\
	\vdots & \ddots & \vdots \\
	Q_{n1} & \dots & Q_{nn} 
	\end{pmatrix}.
	\end{align*}
	Now, observe that
	\begin{align}
	\begin{split}
	\mathbb{E}\exp\left\{ t Z^T Q Z \right\} &= \mathbb{E}\exp\left\{ t \sum_{j=1}^n Z_j ^T Q_{jj} Z_j + t\sum_{j\ne k} Z_j^T Q_{jk}  Z_k  \right\}\\ 
	&\le\exp\left\{ t   \max_{1\le j\le n} \bar b_j^2  {\rm tr}(Q)  \right\} \mathbb{E}\exp\left\{t\sum_{j\ne k}  Z_j^T Q_{jk}  Z_k  \right\},
	\label{c3eqn:lmmim0}
	\end{split}
	\end{align}
	since $ \sum_{j=1}^n \lVert Q_{jj}\rVert_{\rm sp}\le \sum_{j=1}^n {\rm tr}( Q_{jj})= {\rm tr}(Q)$ by its positive semidefiniteness.
	Using the decoupling inequality in Theorem 3.1.1 of \citet{de2012decoupling}, we obtain that 
	\begin{align*}
	\mathbb{E}\exp\left\{t\sum_{j\ne k} Z_j^T  Q_{jk} Z_k  \right\}\le \mathbb{E}\exp\left\{4t\sum_{j=1}^n\sum_{k=1}^n Z_j^T  Q_{jk} \tilde Z_k  \right\},
	\end{align*}
	where $\tilde Z=(\tilde Z_1^T,\dots,\tilde Z_n^T)^T$ is an independent copy of $Z$.
	It is clear that the right-hand side of the display is equal to
	\begin{align}
	\mathbb{E}\mathbb{E}\left[\exp\left\{4t \sum_{k=1}^n \sum_{j=1}^n  Z_j^T Q_{jk} \tilde Z_k  \right\}\Bigg| Z\right]=\mathbb{E}\prod_{k=1}^n\mathbb{E}\left[\exp\left(4t Z^T Q_{\cdot k}  \tilde Z_k \right)\big| Z\right],
	\label{c3eqn:lmmim}
	\end{align}
	where $Q_{\cdot k}=(Q_{1 k}^T ,\dots, Q_{n k}^T)^T\in \mathbb{R}^{ n\times r_k}$.
	Since
	\begin{align*}
	\max_{\xi_k\in{\rm supp}(Z_k)} Z^T Q_{\cdot k}  \xi_k - \min_{\xi_k\in{\rm supp}(Z_k)}Z^T Q_{\cdot k}\xi_k &= \max_{\xi_k,\xi_k'\in{\rm supp}(Z_k)} Z^T Q_{\cdot k} ( \xi_k-\xi_k')\\
	&\le  \left\lVert Q_{\cdot k}^T Z \right\rVert_2 \tilde b_k,
	\end{align*}
	applying Hoeffding's lemma to the inner expectation,
	we bound \eqref{c3eqn:lmmim} by
	\begin{align}
	\mathbb{E}\prod_{k=1}^n\exp\left\{2t^2 \left\lVert Q_{\cdot k}^T Z \right\rVert_2^2 \tilde b_k^2 \right\}\le \mathbb{E}\exp\left\{2t^2 \max_{1\le k\le n}\tilde b_k^2 \sum_{k=1}^n \left\lVert Q_{\cdot k}^T Z \right\rVert_2^2 \right\}.
	\label{c3eqn:lmmim2}
	\end{align}
	Since we have that by the symmetry of $Q$,
	\begin{align*}
	\sum_{k=1}^n \left\lVert Q_{\cdot k}^T Z \right\rVert_2^2 = Z^T\left\{\sum_{k=1}^n  Q _{\cdot k} Q _{\cdot k}^T\right\} Z  = Z^T Q^2 Z,
	\end{align*}
	the right-hand side of \eqref{c3eqn:lmmim2} is bounded by
	\begin{align*}
	\mathbb{E}\exp\left\{2t^2 \max_{1\le k\le n}\tilde b_k^2 Z^T Q^{1/2}Q Q^{1/2} Z\right\}\le  \mathbb{E}\exp\left\{2t^2 \max_{1\le k\le n}\tilde b_k^2 \lVert Q\rVert_{\rm sp} Z^T Q Z\right\},
	\end{align*}
	by the positive semi-definiteness of $Q$. By Jensen's inequality, this is further bounded by
	\begin{align*}
	\left[\mathbb{E}\exp\left\{t Z^T Q Z\right\}\right]^{2t \max_k\tilde b_k^2 \lVert Q\rVert_{\rm sp} },\quad 0<t< \frac{1}{2 \max_k \tilde b_k^2 \lVert Q\rVert_{\rm sp} }.
	\end{align*}
	Combining the last display and \eqref{c3eqn:lmmim0}, we obtain the inequality given in the lemma.
\end{proof}

\bibliographystyle{apalike}
\bibliography{ref}

\begin{thebibliography}{}

\bibitem[Atchad{\'e}, 2017]{atchade2017contraction}
Atchad{\'e}, Y.~A. (2017).
\newblock On the contraction properties of some high-dimensional
  quasi-posterior distributions.
\newblock {\em The Annals of Statistics}, 45(5):2248--2273.

\bibitem[Bach, 2010]{bach2010self}
Bach, F. (2010).
\newblock Self-concordant analysis for logistic regression.
\newblock {\em Electronic Journal of Statistics}, 4:384--414.

\bibitem[Bai et~al., 2020]{bai2019spike}
Bai, R., Moran, G.~E., Antonelli, J.~L., Chen, Y., and Boland, M.~R. (2020).
\newblock Spike-and-slab group lassos for grouped regression and sparse
  generalized additive models.
\newblock {\em Journal of the American Statistical Association}, to appear.

\bibitem[Belitser and Ghosal, 2020]{belitser2017empirical}
Belitser, E. and Ghosal, S. (2020).
\newblock Empirical {B}ayes oracle uncertainty quantification for regression.
\newblock {\em Annals of Statistics}, 48(6):3113--3137.

\bibitem[Blaz{\`e}re et~al., 2014]{blazere2014oracle}
Blaz{\`e}re, M., Loubes, J.-M., and Gamboa, F. (2014).
\newblock Oracle inequalities for a group lasso procedure applied to
  generalized linear models in high dimension.
\newblock {\em IEEE Transactions on Information Theory}, 60(4):2303--2318.

\bibitem[Castillo et~al., 2015]{castillo2015bayesian}
Castillo, I., Schmidt-Hieber, J., and van~der Vaart, A. (2015).
\newblock Bayesian linear regression with sparse priors.
\newblock {\em The Annals of Statistics}, 43(5):1986--2018.

\bibitem[De~la Pena and Gin{\'e}, 2012]{de2012decoupling}
De~la Pena, V. and Gin{\'e}, E. (2012).
\newblock {\em Decoupling: From Dependence to Independence}.
\newblock Springer Science \& Business Media.

\bibitem[Fang et~al., 1990]{fang1990symmetric}
Fang, K.-T., Kotz, S., and Ng, K.-W. (1990).
\newblock {\em Symmetric Multivariate and Related Distributions}.
\newblock Chapman and Hall, London.

\bibitem[Gao et~al., 2020]{gao2015general}
Gao, C., van~der Vaart, A.~W., and Zhou, H.~H. (2020).
\newblock A general framework for {B}ayes structured linear models.
\newblock {\em Annals of Statistics}, 48(5):2848--2878.

\bibitem[Ghosal et~al., 2000]{ghosal2000convergence}
Ghosal, S., Ghosh, J.~K., and van~der Vaart, A.~W. (2000).
\newblock Convergence rates of posterior distributions.
\newblock {\em The Annals of Statistics}, 28(2):500--531.

\bibitem[Ghosal and van~der Vaart, 2007]{ghosal2007convergence}
Ghosal, S. and van~der Vaart, A. (2007).
\newblock Convergence rates of posterior distributions for noniid observations.
\newblock {\em The Annals of Statistics}, 35(1):192--223.

\bibitem[Hoffman and Duncan, 1988]{hoffman1988multinomial}
Hoffman, S.~D. and Duncan, G.~J. (1988).
\newblock Multinomial and conditional logit discrete-choice models in
  demography.
\newblock {\em Demography}, 25(3):415--427.

\bibitem[Hsu et~al., 2012]{hsu2012tail}
Hsu, D., Kakade, S., and Zhang, T. (2012).
\newblock A tail inequality for quadratic forms of subgaussian random vectors.
\newblock {\em Electronic Communications in Probability}, 17:1--6.

\bibitem[Huang and Zhang, 2010]{huang2010benefit}
Huang, J. and Zhang, T. (2010).
\newblock The benefit of group sparsity.
\newblock {\em The Annals of Statistics}, 38(4):1978--2004.

\bibitem[Jeong and Ghosal, 2021a]{jeong2020posterior}
Jeong, S. and Ghosal, S. (2021a).
\newblock Posterior contraction in sparse generalized linear models.
\newblock {\em Biometrika}, 108(2):367--379.

\bibitem[Jeong and Ghosal, 2021b]{jeong2020unified}
Jeong, S. and Ghosal, S. (2021b).
\newblock Unified {B}ayesian theory of sparse linear regression with nuisance
  parameters.
\newblock {\em Electronic Journal of Statistics}, 15(1):3040--3111.

\bibitem[Jiang, 2007]{jiang2007bayesian}
Jiang, W. (2007).
\newblock Bayesian variable selection for high dimensional generalized linear
  models: convergence rates of the fitted densities.
\newblock {\em The Annals of Statistics}, 35(4):1487--1511.

\bibitem[Jin et~al., 2019]{jin2019short}
Jin, C., Netrapalli, P., Ge, R., Kakade, S.~M., and Jordan, M.~I. (2019).
\newblock A short note on concentration inequalities for random vectors with
  subgaussian norm.
\newblock {\em arXiv preprint arXiv:1902.03736}.

\bibitem[Lee and Cao, 2021]{lee2020bayesian}
Lee, K. and Cao, X. (2021).
\newblock Bayesian group selection in logistic regression with application to
  {MRI} data analysis.
\newblock {\em Biometrics}, 77(2):391--400.

\bibitem[Lounici et~al., 2011]{lounici2011oracle}
Lounici, K., Pontil, M., van~de Geer, S., and Tsybakov, A.~B. (2011).
\newblock Oracle inequalities and optimal inference under group sparsity.
\newblock {\em The Annals of Statistics}, 39(4):2164--2204.

\bibitem[Martin et~al., 2017]{martin2017empirical}
Martin, R., Mess, R., and Walker, S.~G. (2017).
\newblock Empirical {B}ayes posterior concentration in sparse high-dimensional
  linear models.
\newblock {\em Bernoulli}, 23(3):1822--1847.

\bibitem[McFadden, 1973]{mcfadden1973conditional}
McFadden, D. (1973).
\newblock Conditional logit analysis of qualitative choice behavior.
\newblock In Zarembka, P., editor, {\em Frontiers in Econometrics}, pages
  105--135. New York: Wiley.

\bibitem[Meier et~al., 2008]{meier2008group}
Meier, L., van~de Geer, S., and B{\"u}hlmann, P. (2008).
\newblock The group lasso for logistic regression.
\newblock {\em Journal of the Royal Statistical Society: Series B (Statistical
  Methodology)}, 70(1):53--71.

\bibitem[Nardi and Rinaldo, 2008]{nardi2008asymptotic}
Nardi, Y. and Rinaldo, A. (2008).
\newblock On the asymptotic properties of the group lasso estimator for linear
  models.
\newblock {\em Electronic Journal of Statistics}, 2:605--633.

\bibitem[Narisetty et~al., 2019]{narisetty2019skinny}
Narisetty, N.~N., Shen, J., and He, X. (2019).
\newblock Skinny {G}ibbs: {A} consistent and scalable {G}ibbs sampler for model
  selection.
\newblock {\em Journal of the American Statistical Association},
  114(527):1205--1217.

\bibitem[Ning et~al., 2020]{ning2018bayesian}
Ning, B., Jeong, S., and Ghosal, S. (2020).
\newblock Bayesian linear regression for multivariate responses under group
  sparsity.
\newblock {\em Bernoulli}, 26(3):2353--2382.

\bibitem[van~de Geer and Muro, 2014]{van2014higher}
van~de Geer, S. and Muro, A. (2014).
\newblock On higher order isotropy conditions and lower bounds for sparse
  quadratic forms.
\newblock {\em Electronic Journal of Statistics}, 8(2):3031--3061.

\bibitem[Vershynin, 2012]{vershynin2012introduction}
Vershynin, R. (2012).
\newblock Introduction to the non-asymptotic analysis of random matrices.
\newblock In Eldar, Y.~C. and Kutyniok, G., editors, {\em Compressed Sensing:
  Theory and Applications}, pages 210--268. Cambridge University Press,
  Cambridge-New York.

\bibitem[Vincent and Hansen, 2014]{vincent2014sparse}
Vincent, M. and Hansen, N.~R. (2014).
\newblock Sparse group lasso and high dimensional multinomial classification.
\newblock {\em Computational Statistics \& Data Analysis}, 71:771--786.

\bibitem[Watson, 1996]{watson1996spectral}
Watson, G.~S. (1996).
\newblock Spectral decomposition of the covariance matrix of a multinomial.
\newblock {\em Journal of the Royal Statistical Society: Series B
  (Methodological)}, 58(1):289--291.

\bibitem[Wei and Ghosal, 2020]{wei2019contraction}
Wei, R. and Ghosal, S. (2020).
\newblock Contraction properties of shrinkage priors in logistic regression.
\newblock {\em Journal of Statistical Planning and Inference}, 207:215--229.

\bibitem[Yuan and Lin, 2006]{yuan2006model}
Yuan, M. and Lin, Y. (2006).
\newblock Model selection and estimation in regression with grouped variables.
\newblock {\em Journal of the Royal Statistical Society: Series B (Statistical
  Methodology)}, 68(1):49--67.

\bibitem[Zajkowski, 2020]{zajkowski2020bounds}
Zajkowski, K. (2020).
\newblock Bounds on tail probabilities for quadratic forms in dependent
  sub-{G}aussian random variables.
\newblock {\em Statistics \& Probability Letters}, 167:108898.

\end{thebibliography}

\end{document}